\documentclass[11pt]{article}
\usepackage{amsfonts}
\usepackage{amsmath,amssymb}
\usepackage[dvips]{graphicx}
\usepackage{algorithm}
\usepackage{algorithmicx}
\usepackage{algpseudocode}
\usepackage{yhmath}
\usepackage{qtree}
\usepackage{xcolor}
\usepackage{epsfig}
\usepackage{lineno,hyperref}
\usepackage{mathtools}

\usepackage{amsthm}

\usepackage[T1]{fontenc}
\usepackage[utf8]{inputenc}
\usepackage{booktabs}
\usepackage{authblk}
\usepackage[toc]{appendix}
\numberwithin{equation}{section}
\usepackage{verbatim}\usepackage{graphicx}
\usepackage{geometry}

\usepackage{amsmath}
\allowdisplaybreaks[1]

\newcommand{\be}{\begin{equation}}
\newcommand{\ee}{\end{equation}}
\newcommand{\ba}{\begin{array}}
\newcommand{\ea}{\end{array}}
\newcommand{\bea}{\begin{eqnarray*}}
\newcommand{\eea}{\end{eqnarray*}}
\newcommand{\bean}{\begin{eqnarray}}
\newcommand{\eean}{\end{eqnarray}}

\newtheorem{theorem}{Theorem}[section]
\newtheorem{lemma}{Lemma}[section]
\newtheorem{remark}{Remark}[section]

\newtheorem{corollary}{Corollary}[section]

\newcommand{\lc}{\mathrel{\raise2pt\hbox{${\mathop<\limits_{\raise1pt\hbox{\mbox{$\sim$}}}}$}}}
\newcommand{\gc}{\mathrel{\raise2pt\hbox{${\mathop>\limits_{\raise1pt\hbox{\mbox{$\sim$}}}}$}}}
\newcommand{\ec}{\mathrel{\raise1pt\hbox{${\mathop=\limits_{\raise2pt\hbox{\mbox{$\sim$}}}}$}}}

\begin{document}

\title{How to Define Dissipation-Preserving Energy for Time-Fractional Phase-Field Equations}

\author[1]{Chaoyu Quan}
\author[2,1]{Tao Tang}
\author[3,1]{Jiang Yang}
\affil[1]{\small SUSTech International Center for Mathematics, Southern University of Science and Technology, Shenzhen, China (\href{mailto:quancy@sustech.edu.cn}{quancy@sustech.edu.cn})}
\affil[2]{\small Division of Science and Technology, BNU-HKBU United International College, Zhuhai, Guangdong, China (\href{mailto:tangt@sustech.edu.cn}{tangt@sustech.edu.cn}).}
\affil[3]{\small Department of Mathematics, Southern University of Science and Technology, Shenzhen, China  (\href{mailto:yangj7@sustech.edu.cn}{yangj7@sustech.edu.cn}).}
\maketitle

\begin{abstract}

There exists a well defined energy for classical phase-field equations under which the dissipation law is satisfied, i.e., the energy is non-increasing with respect to time. However, it is not clear how to extend the energy definition to time-fractional phase-field equations so that the corresponding dissipation law is still satisfied. 
In this work, we will try to settle this problem for phase-field equations with Caputo time-fractional derivative, by defining a nonlocal energy as an averaging of the classical energy with a time-dependent weight function.
As the governing equation exhibits both nonlocal and nonlinear behavior, the dissipation analysis is challenging. 
To deal with this, we propose a new theorem on judging the positive definiteness of a  symmetric function, that is derived from a special Cholesky decomposition.
Then, the nonlocal energy is proved to be dissipative under a simple restriction of the weight function.
Within the same framework, the time fractional derivative of classical energy for time-fractional phase-field models can be proved to be always nonpositive.

\end{abstract}


{\bf AMS: }
65M06, 65M12, 74A50

\section{Introduction}

A fractional time derivative arises when the characteristic waiting time diverges, which  models situations involving memory, see, e.g.,  \cite{MK2000,Za2002}. 
In recent years, to model memory effects and subdiffusive regimes in applications such as transport theory, viscoelasticity, rheology and non-Markovian stochastic processes, there has been an increasing interest in the study of time-fractional differential equations, i.e., differential equations where the standard time derivative is replaced by a fractional one, typically a Caputo or a Riemann-Liouville derivative.

For the models involved Caputo fractional derivative, Allen, Caffarelli and Vasseur \cite{Allen2016APP} studied the regularity of a time-fractional parabolic problem. Their main result is a De Giorgi-Nash-Moser H{\"o}lder regularity theorem for solutions in a divergence form equation. In a more recent work \cite{Allen2017PorousMF}, they performed regularity study for
porous medium flow with both a fractional potential pressure and fractional time derivative. In \cite{Luchko2017}, Luchko and Yamamot  discussed the maximum principle for a class of time-fractional diffusion equation with the Caputo time-derivative. In \cite{LiuJG2018}, Li, Liu and Wang investigated Cauchy problems for nonlinear time-fractional Keller-Segel equation  with the Caputo time-derivative. Some important properties of the solutions including the nonnegativity preservation, mass conservation and blowup behaviors are established.

On the other hand, for the models involved Riemann-Liouville fractional time derivative,
 Zach \cite{Zacher2013}  investigated the regularity of weak solutions to a class of time fractional diffusion equations and  obtained a De Giorgi-Nash type theorem which gives an interior H{\"o}lder estimate for bounded weak solutions. In  \cite{VZacher2013}, Vergara and Zacher  investigated optimal decay estimates by using energy methods; and in  \cite{VZacher2017}, they studied instability and blowup properties for Riemann-Liouville time-fractional subdiffusion equations. In \cite{LeWM2019}, Le,  McLean and Stynes studied the well-posedness of the solution of the time-fractional Fokker-Planck equation
with general forcing.

Most of the works mentioned above are of semi-linearity in space. It is noticed that there exists active research on time-fractional problems with spatial nonlinearity, which arises in practical applications. For example, Allen, Caffarelli and Vasseur \cite{Allen2017PorousMF} considered a time-space fractional porous medium equation with Caputo fractional time derivatives and nonlocal diffusion effects. In \cite{Giga2017}, Giga and Namba investigated the well-posedness of Hamilton-Jacobi equations with a Caputo fractional time derivative, with a main purpose of finding a proper notion of viscosity solutions so that the underlying Hamilton-Jacobi equation is well-posed. A further study along this line
 is recently provided by Camilli and Goffi \cite{Camilli2019}. Their study relies on a combination of a gradient bound for the time-fractional Hamilton-Jacobi equation obtained via nonlinear adjoint method and sharp estimates in Sobolev and H{\"o}lder spaces for the corresponding linear problem.

 The Cahn--Hilliard model \cite{cahn1958free} may be the most popular phase-field model whose governing equation is of the form
\begin{equation}
  \label{eq:Cahn--Hilliard}
    \partial_t \phi + \gamma(- \Delta)
    \left( - \varepsilon^2  \Delta \phi + F'(\phi)
    \right ) = 0, \quad x \in \Omega \subset \mathbb{R}^d, ~ 0 < t \le T,
\end{equation}
where $\varepsilon$ is an interface width parameter,
$\gamma$ is the mobility, and $F$ is a double-well potential that is usually taken the form
$F(\phi) = \frac{1}{4}(1 -\phi^2)^2$. The corresponding free energy functional for the Cahn--Hilliard equation (\ref{eq:Cahn--Hilliard}) is defined as
\begin{equation}
  \label{eq:energy0}
  E(\phi):= \int_{\Omega} \Bigl(\frac{\varepsilon^2}2 |\nabla \phi|^2 +  F(\phi)\Bigr)  \,{\rm d}x.
\end{equation}
The Cahn--Hilliard  equation can be viewed as a gradient flow with the energy (\ref{eq:energy0}) in $H^{-1}$. It is well known that with proper boundary conditions the energy functional $E$ decreases in time:
\begin{equation}
  \label{eq:energy-law_II}
 \frac{\rm d}{{\rm d}t} E(\phi) =
  -\int_\Omega \left|\nabla \left( - \varepsilon^2 \Delta \phi
  + F'(\phi) \right)\right|^2 {\rm d}x \leq 0.
\end{equation}
This dissipation law has been used extensively as the nonlinear numerical stability criteria.

The present paper is concerned with time-fractional phase-field equations.
 Without loss of generality, we consider the most representative phase-field models, i.e., the Allen-Chan model \cite{allen1979microscopic} and the Cahn--Hilliard  model \cite{cahn1958free}:
\begin{equation} \label{eq:phase_field}
  \begin{array}{r@{}l}
 	 \left\{
	\begin{aligned}
	&\partial^\alpha_t \phi = \gamma \, \mathcal G \left(-\varepsilon^2 \Delta \phi +  F'(\phi) \right) && \mbox{in } \Omega\times (0,T], \\
	& \phi(x,0) = \phi_0(x)  && \mbox{in } \Omega,
	\end{aligned}
	\right.
  \end{array}
\end{equation}
where $\alpha \in (0,1)$, $\varepsilon>0$ is the interface width parameter, $\gamma > 0$ is the mobility constant, $F$ is a double-well potential functional, and $\mathcal G = -1$ (Allen--Cahn) or $\Delta$ (Cahn--Hilliard) is a nonpositive operator.
Here, the Caputo fractional derivative of $\phi$ is given by
\begin{equation}\label{eq:Caputo}
\partial^\alpha_t \phi(t) \coloneqq \frac{1}{\Gamma(1-\alpha)}\int_0^t \frac{\phi'(s)}{(t-s)^\alpha}\, {\rm d} s, \quad t\in(0,T),
\end{equation}
where $\Gamma(\cdot)$ is the gamma function. For simplicity, a periodic boundary condition is assumed.

The main purpose of this work is to extend the energy definition (\ref{eq:energy0}) from the classical phase-field models to the time-fractional models (\ref{eq:phase_field}), with the requirement that the energy is decreasing with time.
To do this, we consider a weighted energy $E_\omega(t)$ in the following form:
\begin{equation}
\label{eq:newenergy0}
	E_\omega(t)=  \int_0^1 \omega(\theta) E(\theta t) \, {\rm d} \theta, 
\end{equation}
where $\omega(\cdot)\geq 0$ is some weight function satisfying $\int_0^1 \omega(\theta) {\rm d} \theta = 1$ and $E(\theta t)=E(\phi(\cdot,\theta t))$ is the classical energy defined by \eqref{eq:energy0}. 
Note that $E_\omega$ is a nonlocal energy.
We prove that if $\omega(\theta) \theta^{1-\alpha}(1-\theta)^\alpha$ is nonincreasing w.r.t. $\theta$, then
\begin{equation} \label{mainres0}
\frac{\rm d}{{\rm d}t}E_\omega(t) \leq 0, \quad \forall\;0<t <T.
\end{equation}
In fact, the above result can be achieved as soon as we can prove the negativeness of a special integral involving a weakly singular function.
To do this, we introduce a special Cholesky decomposition, which leads to a new way on judging the positive definiteness of a kernel.
Then, we can show that \eqref{mainres0} holds as long as $\omega(\theta)\theta^{1-\alpha}(1-\theta)^\alpha$ is nonincreasing.

Furthermore, another interesting dissipation result can be obtained from similar analysis.
More precisely,  in the spirit of \eqref{eq:Caputo}, we can define the Caputo time-fractional derivative of classical energy in the following sense
\begin{equation}\label{eq:CaputoEnergy}
\partial^\alpha_t E(t) \coloneqq \frac{1}{\Gamma(1-\alpha)}\int_0^t \frac{E'(s)}{(t-s)^\alpha}\, {\rm d} s, \quad t\in(0,T),
\end{equation}
where $E(s)=E(\phi(\cdot,s))$ is given by \eqref{eq:energy0}. 
In this work, we will show that the time-fractional derivative of classical energy \eqref{eq:CaputoEnergy} is always nonpositive, i.e., 
\begin{equation}\label{eq:fractional_dissipation}
\partial_t^\alpha E(t) \leq 0, \quad \forall\; 0<t<T,
\end{equation}
which was observed in previous numerical simulations \cite{du2019time}, but theoretical proof was not provided.

The paper is organized as follows. 
In Section \ref{sect2}, we first introduce a useful lemma relevant to Cholesky decomposition and then give a theorem on judging the positive definiteness of a kernel. 
In Section \ref{sect3}, the main theorem on energy dissipation result \eqref{mainres0} will be established.
In Section \ref{sect4}, we prove that the fractional derivative of classical energy \eqref{eq:energy0} is always nonpositive.  
Some concluding remarks will be provided in the final section.

\section{A result on positive definite kernel}
\label{sect2}

Before introducing the theorem on the positive definite kernel, we propose a useful lemma about a special Cholesky decomposition.

\begin{lemma}[A special Cholesky decomposition]\label{lem:1}
Given an arbitrary symmetric matrix $\mathbf S$ of size $N\times N$ with positive elements. If $\mathbf S$ satisfies the following properties:
\begin{itemize}\label{item:1}
\item[{\rm (P1):}] $\forall\; 1\leq j < i \leq N$, $\left[ \mathbf S \right]_{i-1,j}\geq \left[ \mathbf S \right]_{i, j}$;
\item[{\rm (P2):}] $\forall\; 1 < j \leq i \leq N$, $\left[ \mathbf S \right]_{i, j-1}< \left[ \mathbf S \right]_{i, j}$;
\item[{\rm (P3):}] $\forall \;1< j < i \leq N$, $\left[ \mathbf S \right]_{i-1, j-1} - \left[ \mathbf S \right]_{i, j-1}\leq \left[ \mathbf S \right]_{i-1, j} - \left[ \mathbf S \right]_{i, j}$,
\end{itemize}
then ${\mathbf S}$ is positive definite.
In particular, ${\mathbf S}$ has the following Cholesky decomposition:
\begin{equation}
{\mathbf S} = {\mathbf L} {\mathbf L}^{\rm T},
\end{equation}
where ${\mathbf L}$ is a lower triangular matrix satisfying
\begin{itemize}
\item[{\rm (Q1):}] $\forall 1\leq j\leq i\leq N$, $\left[ \mathbf L \right]_{ij} > 0$;
\item[{\rm (Q2):}] $\forall 1\leq j < i \leq N$, $\left[ \mathbf L \right]_{i-1,j}\geq \left[ \mathbf L \right]_{i, j}$.
\end{itemize}
\end{lemma}

\begin{proof}
Let $\mathbf S_n$ be the $n$th principal submatrix of $\mathbf S$ of size $n\times n$ with $n\leq N$.
We will give a proof by induction on $n$. 

{\it \underline{Initial case:}}
First, we need to check the case of $\mathbf S_2$.
Obviously, we have the following Cholesky decomposition:
\begin{equation}
\mathbf S_{2} = \left[\begin{array}{cc} {[\mathbf S]}_{11} & {[\mathbf S]}_{12} \\ 
{[\mathbf S]}_{12} & {[\mathbf S]}_{22} \end{array}\right] = 
\left[\begin{array}{cc} \sqrt{{[\mathbf S]}_{11}} & 0 \\ 
\frac{{[\mathbf S]}_{12}}{\sqrt{{[\mathbf S]}_{11}}} & \sqrt{ {[\mathbf S]}_{22} - \frac{{[\mathbf S]}_{12}^2}{{[\mathbf S]}_{11}}}  \end{array}\right]
\left[\begin{array}{cc} \sqrt{{[\mathbf S]}_{11}} & \frac{{[\mathbf S]}_{12}}{\sqrt{{[\mathbf S]}_{11}}} \\ 
0 & \sqrt{ {[\mathbf S]}_{22} - \frac{{[\mathbf S]}_{12}^2}{{[\mathbf S]}_{11}}} \end{array}\right],
\end{equation}
where ${[\mathbf S]}_{12} \leq {[\mathbf S]}_{11} $ and ${[\mathbf S]}_{12}< {[\mathbf S]}_{22}$. 
It is easy to find that the lower triangular matrix in the above decomposition satisfies the properties (Q1) and (Q2). 

{\it \underline{Inductive step:}}
Assume that Lemma \ref{lem:1} is always true until $\mathbf S_{n-1}$ which can be decomposed as $\mathbf L_{n-1} \mathbf L_{n-1}^{\rm T}$.
We will show that $\mathbf S_n$ can still be decomposed as $\mathbf L_n \mathbf L_n^{\rm T}$, with the lower triangular matrix $\mathbf L_n$ satisfying the properties (Q1) and (Q2).

We split $\mathbf S_n$ as follows:
\begin{equation}\label{eq:Bn}
\mathbf S_{n} = \left[\begin{array}{cc} \mathbf S_{n-1} & \mathbf b \\ \mathbf b^{\rm T} & [\mathbf S]_{n,n} \end{array}\right],
\end{equation}
where the $j$th entry of the column vector $\mathbf b$ is $\left[\mathbf b\right]_j  =  [\mathbf S]_{n,j}, 1\le j\le n-1$. We need to find an ${\mathbf L_n}$ such that $\mathbf L_n \mathbf L_n^{\rm T}= \mathbf S_n$. Assume ${\mathbf L_n}$ is of the following form
\begin{equation}
\mathbf L_n
=\left[\begin{array}{cc} \mathbf L_{n-1} & \mathbf 0 \\  \mathbf l^{\rm T} & l_{n,n}\end{array}\right]
= \left[\begin{array}{ccccc} l_{1,1} &  &  &  & \\ l_{2,1} & l_{2,2} &  &  & \\ \vdots & \vdots & \ddots &  & \\ l_{n-1,1} & l_{n-2,2} & \cdots & l_{n-1,n-1} & \\ l_{n,1} & l_{n,2} & \cdots & l_{n,n-1} & l_{n,n} \end{array}\right].
\end{equation}
It follows from the splitting form \eqref{eq:Bn} of $\mathbf S_n$ and $\mathbf L_n \mathbf L_n^{\rm T}= \mathbf S_n$ that $\mathbf l^{\rm T} = \left(l_{n,1},\cdots,l_{n,n-1}\right)$ should satisfy
\begin{equation}\label{eq:Ll}
\mathbf L_{n-1} \mathbf l  = \mathbf b
\end{equation}
and $l_{n,n}$ should satisfy
\begin{equation}\label{eq:LN}
\mathbf l^{\rm T} \mathbf l + l_{n,n}^2 = [\mathbf S]_{n,n}.
\end{equation}
We need to prove that the solution $(\mathbf l, l_{n,n})$ to \eqref{eq:Ll} and \eqref{eq:LN} exists and satisfies
\begin{equation}
0< l_{n,j} \leq l_{n-1,j}, \;\;\; 1\leq j\leq n-1; \quad \;
 l_{n,n} > 0.  \label{eq:lnn}
\end{equation}

We now prove the first part of \eqref{eq:lnn} by induction.
When $j=1$, according to \eqref{eq:Ll} and the property (P1)  of $\mathbf S$, we have
\begin{equation}
0 < l_{n,1} = \frac{[\mathbf S]_{n,1}}{l_{1,1}} \le \frac{[\mathbf S]_{n-1,1}}{l_{1,1}}= l_{n-1,1} ,
\end{equation}
meaning that the first part of \eqref{eq:lnn} is true for $j = 1$.
Assume that the first part of \eqref{eq:lnn} holds for any $1\leq j \leq m$ with $1\leq m< n-1$,
we want to prove that it is also true for $j=m+1$, i.e.,
\begin{equation}\label{eq:lnm1}
0< l_{n,m+1} \leq l_{n-1,m+1}.
\end{equation}
In fact, from \eqref{eq:Ll}, we know that
\begin{eqnarray}
&& [\mathbf S]_{n,m} = \sum_{j=1}^m l_{n,j} \, l_{m,j}, \label{eq:lem_chol1}
\\
&& [\mathbf S]_{n,m+1} = \sum_{j=1}^{m+1} l_{n,j} \, l_{m+1,j} . \label{eq:lem_chol2}
\end{eqnarray}
Subtracting \eqref{eq:lem_chol1} from \eqref{eq:lem_chol2}, according to the property (P2) of $\mathbf S$, we have
\[
0 <  [\mathbf S]_{n,m+1} -[\mathbf S]_{n,m}  \leq \sum_{j=1}^m l_{n,j} \, (l_{m+1,j}-l_{m,j}) + l_{n,m+1}\, l_{m+1,m+1}.
\]
Since $l_{m+1,j}-l_{m,j} \leq 0$, $\forall 1\leq j\leq m$ and $l_{m+1,m+1}> 0$, we deduce from the above inequality that
\begin{equation}\label{eq:lm1_0}
l_{n,m+1}>0.
\end{equation}
Similar to \eqref{eq:lem_chol1} and \eqref{eq:lem_chol2}, we also have
\begin{eqnarray}
&& [\mathbf S]_{n-1,m} = \sum_{j=1}^m l_{n-1,j} \, l_{m,j}, \label{eq:lem_chol3} \\
&& [\mathbf S]_{n-1,m+1} = \sum_{j=1}^{m+1} l_{n-1,j} \, l_{m+1,j} . \label{eq:lem_chol4}
\end{eqnarray}
Subtracting \eqref{eq:lem_chol1} from \eqref{eq:lem_chol3}, we obtain
\begin{equation}\label{eq:Bn1m}
[\mathbf S]_{n-1,m} - [\mathbf S]_{n,m} =  \sum_{j=1}^m (l_{n-1,j}-l_{n,j}) \, l_{m,j},
\end{equation}
and subtracting  \eqref{eq:lem_chol2} from \eqref{eq:lem_chol4}, we obtain
\begin{equation}\label{eq:Bn1m1}
[\mathbf S]_{n-1,m+1} - [\mathbf S]_{n,m+1} =  \sum_{j=1}^{m+1} (l_{n-1,j}-l_{n,j}) \, l_{m+1,j}.
\end{equation}
Combining \eqref{eq:Bn1m}, \eqref{eq:Bn1m1}, and the property (P3) of $\mathbf S$, we then have
\[
  \begin{array}{r@{}l}
	\begin{aligned}
	0
	& \leq \left([\mathbf S]_{n-1,m+1} - [\mathbf S]_{n,m+1} \right) - \left([\mathbf S]_{n-1,m} - [\mathbf S]_{n,m} \right) \\
 	& = \sum_{j=1}^{m} (l_{n-1,j}-l_{n,j}) \, (l_{m+1,j}-l_{m,j}) + (l_{n-1,m+1}-l_{n,m+1}) \, l_{m+1,m+1}
 	\end{aligned}.
  \end{array}
\]
Since $l_{n-1,j}-l_{n,j} \geq 0$, $l_{m+1,j}-l_{m,j}\leq 0$, $\forall 1\leq j\leq m$, and $ l_{m+1,m+1}>0$,  we then obtain from the above inequality that
\begin{equation}
l_{n,m+1}  \leq l_{n-1,m+1}.
\end{equation}
Combining this inequality with \eqref{eq:lm1_0}, we obtain that \eqref{eq:lnm1} is true where $j= m+1$.
By induction, we conclude that the first part of \eqref{eq:lnn} holds for any $1\leq j\leq n-1$.

We now turn to prove the second part of \eqref{eq:lnn}, i.e., $l_{n,n} >0$. It follows from \eqref{eq:Ll} and the first part of \eqref{eq:lnn}, we have
\begin{equation}\label{eq:Snn1}
[\mathbf S]_{n,n-1}  = \sum_{j=1}^{n-1} l_{n,j} l_{n-1,j} \ge  \sum_{j=1}^{n-1} l_{n,j}^2
\end{equation}
and using \eqref{eq:LN} gives
\[
[\mathbf S]_{n,n} = \sum_{j=1}^{n} l_{n,j}^2.
\]
This, together with \eqref{eq:Snn1}, gives
\begin{equation}
 l_{n,n}^2 \geq [\mathbf S]_{n,n} - [\mathbf S]_{n,n-1} > 0,
\end{equation}
where the property (P2) is used.
This implies that $l_{n,n}$ is a real number and we can take $l_{n,n}>0$.

In summary, we have proved that $\mathbf L_n$ is computable and satisfies \eqref{eq:lnn}.
Therefore, $\mathbf L_n$ satisfies the properties (Q1) and (Q2) in the lemma and the principal submatrix $\mathbf S_n$ is then positive definite.

{\it \underline{Conclusion:}}
By induction, we conclude that the lemma holds for the full matrix $\mathbf S$ of size $N\times N$ with $N\in \mathbb N_+$.
\end{proof}

From the kernel point of view, the special Cholesky decomposition provides a new way on judging if a symmetric positive function is a positive definite kernel. 
We state and prove the related theorem below. 

\begin{theorem}\label{thm:kernel}
Given a symmetric function $\kappa(x,y)> 0$ defined on $\mathbb R^2$. 
If $\kappa(x,y)$ satisfies
\begin{itemize}
\item $\partial_x \kappa(x,y)\leq 0,~\forall x > y$;
\item $\partial_y \kappa(x,y) > 0,~\forall x > y$;
\item $\partial_{xy} \kappa(x,y) \leq 0,~\forall x > y$,
\end{itemize}
then $\kappa(x,y)$ is a positive definite kernel. 
\end{theorem}

\begin{proof}
Take an arbitrary sequence of points $x_1,x_2,\cdots,x_N\in \mathbb R$, $N\in \mathbb N_+$.
Without loss of generalization, we assume that $x_1<\ldots<x_N$. 
Then, $\forall c_1,\ldots,c_N\in \mathbb R$, we want to prove
\begin{equation}\label{ineq:PD}
\sum_{i=1}^N\sum_{j=1}^N c_i c_j \kappa(x_i,x_j) \geq 0.
\end{equation}
Let $\mathbf K = \left[\kappa(x_i,x_j)\right]_{N\times N}$ be the symmetric matrix corresponding to the left-hand side of the above inequality. 
From the three conditions of $\kappa$ in this theorem, it is not difficult to verify that $\mathbf K$ satisfies all three properties in Lemma \ref{lem:1}.
In particular, straight computation gives: $\forall i<j$,
\begin{equation}
  \begin{array}{r@{}l}
	\begin{aligned}
	& \left( \left[ \mathbf K \right]_{i-1, j-1} - \left[ \mathbf K \right]_{i, j-1} \right) - \left( \left[ \mathbf K \right]_{i-1, j} - \left[ \mathbf K \right]_{i, j} \right) \\
	& = \kappa(x_{i-1},x_{j-1}) - \kappa(x_{i},x_{j-1}) - \kappa(x_{i-1},x_j) + \kappa(x_i,x_j) \\
	& = \int_{x_{j-1}}^{x_j} \int_{x_{i-1}}^{x_i} \partial_{xy} \kappa(x,y) \,{\rm d}x\,{\rm d}y
	\leq 0,
 	\end{aligned}
  \end{array}
\end{equation}
that is the third property in Lemma \ref{lem:1}.
Therefore, $\mathbf K$ is a positive definite matrix and the inequality \eqref{ineq:PD} always holds, meaning that $\kappa(x,y)$ is a positive definite kernel.
\end{proof}

\begin{remark}\label{rmk:kernel}
If $\kappa(x,y)$ is a positive definite kernel, then $\kappa(y,x)$ is also a positive definite kernel.
This indicates that if  a positive symmetric function $\kappa$ satisfies $\partial_x \kappa<0$, $\partial_y\kappa \geq 0$, and $\partial_{xy} \kappa\leq 0$ for all $x>y$, then $\kappa$ is a positive definite kernel as in Theorem \ref{thm:kernel}. 
\end{remark}

\begin{remark}\label{rmk:kernel2}
The well-known Abel kernel $e^{-\left|x-y\right|}$ satisfies the three properties in Theorem \ref{thm:kernel} and is consequently a positive definite kernel.
\end{remark}

\section{Dissipation-preserving energy}
\label{sect3}
In this section, we shall construct a dissipation-preserving energy based on the result in Theorem \ref{thm:kernel}.
Consider the classical energy functional for the time-fractional Allen--Cahn or Cahn--Hilliard equation \eqref{eq:phase_field}:
\begin{equation}\label{eq:energy}
E(t) = \int_\Omega \left(\frac{\varepsilon^2} 2 \left| \nabla \phi \right|^2 + F(\phi) \right) \, {\rm d} x.
\end{equation}
Straightforward computation of its derivative with respect to time gives
\begin{equation}\label{eq:ed_acch}
E'(t)= \int_\Omega  \partial_t \phi  \left( - \varepsilon^2 \Delta \phi +  F'(\phi)\right) {\rm d}x
= \frac 1 \gamma \int_\Omega \partial_t \phi\left({\mathcal G^{-1}}\partial^\alpha_t \phi \right) {\rm d}x,
\end{equation}
where $\mathcal G^{-1}$ is the inverse of $\mathcal G$.
It is still a challenge to prove $E'(t)\leq 0$ despite that numerous numerical tests have verified this. 
We remark that Tang et al. demonstrated in \cite{TangYZ19} that the energies associated with the time-fractional problems are bounded above by the initial energy, i.e., 
\begin{equation}
\label{eq:tyz}
E(t)\le E(0),\quad \text{for all} \quad t>0.
\end{equation}

To preserve the dissipation law, we consider a weighted energy $E_\omega(t)$ in the form of
\begin{equation}
\label{eq:newenergy}
	E_\omega(t)=  \int_0^1 \omega(\theta) E(\theta t) \, {\rm d} \theta, 
\end{equation}
where $\omega(\cdot)\geq 0$ is some weight function satisfying $\int_0^1 \omega(\theta) {\rm d} \theta = 1$.
It is then followed from (\ref{eq:newenergy}) that
\begin{equation}
 E_\omega (t) \le \int_0^1 \omega(\theta) E(0) \, {\rm d} s
 = E(0), \quad \forall \; t>0.
 \end{equation}
This indicates that $E_\omega$ is also bounded by the initial energy.
Further, it follows from \eqref{eq:tyz} and \eqref{eq:newenergy} that
\begin{eqnarray}
	 E'_\omega(t) =  \int_0^1  \omega(\theta) \theta E'(\theta t) \, {\rm d} \theta.
    \label{eq:Ea_der0}
\end{eqnarray}
Substituting \eqref{eq:ed_acch} into \eqref{eq:Ea_der0} and taking into account the periodic boundary condition, we have
\begin{equation}\label{eq:Ea_der}
  \begin{array}{r@{}l}
	\begin{aligned}
	E_\omega'(t) 
   	 & = - \frac{t^{1-\alpha}}{\gamma \Gamma(1-\alpha)}  \int_\Omega \int_0^1 \int_0^\theta \frac{\omega(\theta) \theta}{(\theta-\eta)^\alpha}   \psi(\theta t)\psi(\eta t) \, {\rm d} \eta \, {\rm d} \theta \, {\rm d} x  \\
	 &  =  - \frac{t^{1-\alpha}}{2\gamma \Gamma(1-\alpha)}  \int_\Omega \int_0^1 \int_0^1 \kappa(\theta,\eta)   \psi(\theta t)\psi(\eta t) \, {\rm d} \eta \, {\rm d} \theta \, {\rm d} x, 
	\end{aligned}
   \end{array}
\end{equation}
where 
\begin{equation}\label{eq:psi}
\psi =  \left\{
  \begin{array}{r@{}l}
	\begin{aligned}
   	 & \phi' && \mbox{Allen--Cahn,} \\
	 & \nabla (-\Delta)^{-1}\phi' && \mbox{Cahn--Hilliard,}
	\end{aligned}
   \end{array}
   \right.
\end{equation}
and 
\begin{equation}\label{eq:kappa}
\kappa(\theta,\eta) =  \left\{
  \begin{array}{r@{}l}
	\begin{aligned}
		& \frac{\omega(\theta) \theta}{(\theta-\eta)^\alpha}  &&  \theta > \eta, \\
		& \frac{\omega(\eta) \eta}{(\eta-\theta)^\alpha} &&  \theta < \eta.
	\end{aligned}
   \end{array}
   \right.
\end{equation}
We assume that the solution $\phi$ is first-order continuously differentiable w.r.t. time. 
As soon as $\kappa(\theta,\eta)$ is a positive definite kernel, the dissipation property of $E_\omega$ will be ensured, i.e., $E'_\omega(t)\leq 0$.
Based on Theorem \ref{thm:kernel}, we state and prove the following theorem on the  dissipation-preserving energy. 

\begin{theorem} \label{main-thm1}
For the Allen--Cahn and Cahn--Hilliard models \eqref{eq:phase_field}, if function $\omega(\theta) \theta^{1-\alpha}(1-\theta)^\alpha$ is nonincreasing w.r.t. $\theta$, then the weighted energy \eqref{eq:newenergy} is dissipative, i.e., $E'_\omega(t) \leq 0,~ \forall t >0$. \end{theorem}
\begin{proof}
When $\theta>\eta$, $\kappa(\theta,\eta)$ given by \eqref{eq:kappa} can be rewritten as
\begin{equation}
\kappa(\theta,\eta) 
= \omega(\theta) \theta^{1-\alpha} (1-\theta)^\alpha \frac{\theta^\alpha(1-\eta)^\alpha}{(\theta-\eta)^\alpha} \frac{1}{(1-\theta)^\alpha (1-\eta)^\alpha}.
\end{equation}
It is trivial to see that $\frac{1}{(1-\theta)^\alpha (1-\eta)^\alpha}$ is a positive definite kernel. 
Further, one can easily verify that
\begin{equation}
\mu(\theta,\eta) = \frac{\theta^\alpha(1-\eta)^\alpha}{(\theta-\eta)^\alpha}, \quad \forall \theta > \eta
\end{equation}
decreases w.r.t $\theta$, while increases w.r.t. $\eta$.
Moreover, straight computation gives
\begin{equation}
  \begin{array}{r@{}l}
	\begin{aligned}
	\partial_{\theta\eta} \mu(\theta,\eta)
	& = \partial_\eta\left[ \alpha(1-\eta)^\alpha \left( \theta^{\alpha-1} (\theta-\eta)^{-\alpha} - \theta^\alpha(\theta-\eta)^{-\alpha-1}\right) \right] \\
	& = \alpha^2(1-\eta)^{\alpha-1}  \theta^{\alpha-1} (\theta-\eta)^{-\alpha-1} \eta
	-\alpha^2 (1-\eta)^\alpha \theta^{\alpha-1}(\theta-\eta)^{-\alpha-2}(\alpha\theta+\eta) \\
	& = -\alpha^2 (1-\eta)^{\alpha-1}\theta^{\alpha-1}(\theta-\eta)^{-\alpha-2} \left[ \eta (1-\theta)+\alpha\theta(1-\eta) \right]\\
	& \leq 0.
	\end{aligned}
   \end{array}
\end{equation}
Since $\omega(\theta) \theta^{1-\alpha}(1-\theta)^\alpha$ is nonincreasing, $\omega(\theta) \theta^{1-\alpha}(1-\theta)^\alpha \mu(\theta,\eta)$ satisfies the three conditions in Theorem \ref{thm:kernel}.
Therefore, its symmetric extension is a positive definite kernel.

In summary, $\kappa(\theta,\eta)$ in \eqref{eq:kappa} is the product of two positive definite kernels and itself is consequently a positive kernel.
Therefore, we have $E'_\omega(t)\leq 0$ according to \eqref{eq:Ea_der}.
\end{proof}

\begin{corollary}
Consider the following two cases:
\begin{itemize}
\item[(i)] 
\begin{equation}
\omega(\theta) = \frac{1}{B(\alpha,1-\alpha)\theta^{1-\alpha} (1-\theta)^\alpha},
\end{equation}
where $B(\cdot,\cdot)$ is the Beta function, and
\item[(ii)]
\begin{equation}
\omega(\theta) = \frac{1}{\alpha\theta^{1-\alpha} },
\end{equation}
\end{itemize}
it can be verified that the weighted energy $E_\omega$ in \eqref{eq:newenergy} is dissipative for both cases.
\end{corollary}

\section{Fractional derivative of classical energy}\label{sect4}
We have discussed how to construct a weighted energy for the time-fractional phase-field equations,  which preserves the dissipation law, i.e., $E'_\omega(t)\leq 0$ for all $t>0$. 
However, it is still an open question if $E'(t) \leq 0$ holds true. 
We don't have an affirmative answer yet. 
But from another point of view, we can show that the dissipation of classical energy \eqref{eq:energy} holds in the sense of time-fractional derivative.

\begin{theorem}\label{main-thm2}
 For the Allen--Cahn and Cahn--Hilliard models \eqref{eq:phase_field}, the Caputo time-fractional derivative of the classical energy is always nonpositive, i.e., \eqref{eq:fractional_dissipation} holds.
\end{theorem}

\begin{proof}
Substituting \eqref{eq:ed_acch} into \eqref{eq:CaputoEnergy} yields
\begin{equation}
  \begin{array}{r@{}l}
	\begin{aligned}
	\partial_t^\alpha E(t)  
	& = - \frac 1 {\gamma\Gamma(1-\alpha)^2} \int_\Omega \int_0^t \int_0^s \frac{\psi(s)\psi(\tau)}{(t-s)^\alpha (s-\tau)^\alpha} \, {\rm d} \tau \, {\rm d} s \, {\rm d} x \\
	& =  - \frac 1 {2\gamma\Gamma(1-\alpha)^2} \int_\Omega \int_0^t \int_0^t \kappa(s,\tau) \psi(s)\psi(\tau) \, {\rm d} \tau \, {\rm d} s \, {\rm d} x,
	\end{aligned}
   \end{array}
\end{equation}
where $\psi$ is given by \eqref{eq:psi} and
\begin{equation}\label{eq:kappa2}
\kappa(s,\tau) =  \left\{
  \begin{array}{r@{}l}
	\begin{aligned}
		& \frac{1}{(t-s)^\alpha(s-\tau)^\alpha}  &&  s > \tau, \\
		& \frac{1}{(t-\tau)^\alpha(\tau-s)^\alpha} &&  s < \tau.
	\end{aligned}
   \end{array}
   \right.
\end{equation}
When  $s>\tau$, we can rewrite
\begin{equation}
\kappa(s,\tau) = \frac{1}{(t-s)^\alpha(t-\tau)^\alpha} \frac{(t-\tau)^\alpha}{(s-\tau)^\alpha}.
\end{equation}
It is trivial to see that $\frac{1}{(t-s)^\alpha(t-\tau)^\alpha}$ is a positive definite kernel.
Further, we can find easily that
\begin{equation}
\mu(s,\tau) = \frac{(t-\tau)^\alpha}{(s-\tau)^\alpha}
\end{equation} 
decreases w.r.t. $s$, while increases w.r.t. $\tau$. 
Straight computation gives 
\begin{equation}
  \begin{array}{r@{}l}
	\begin{aligned}
	\partial_{s\tau}\mu(s,\tau) 
	& = \partial_\tau \left[ - \alpha(t-\tau)^\alpha \left(s-\tau\right)^{-\alpha-1} \right] \\
	& = -\alpha (t-\tau)^{\alpha-1}(s-\tau)^{-\alpha-2} \left[  (t-\tau)+\alpha(t-s) \right]\\
	& \leq 0.
	\end{aligned}
   \end{array}
\end{equation}
According to Theorem \ref{thm:kernel}, the symmetric expansion of $\mu(s,\tau)$ is a positive definite kernel.
Therefore, $\kappa(s,\tau)$ in \eqref{eq:kappa2} is a positive definite kernel.
This means that $\partial_t^\alpha E(t)\leq 0$ for all $t>0$.
\end{proof}

%
%
%

\section{Conclusion}

It is known that the historic memory of time-fraction plays a significant role as demonstrated in many numerical simulations, see, e.g., \cite{TangYZ19,WangH2018,Plociniczak2018}. Although the whole evolution process may be slower due to the memory effect, it is still expected that main regularity properties, nonlinear stability and other main features of the relevant phase-field  equations will be preserved. 
The main purpose of this work is along this direction. 
More specifically, we have proposed a new energy $E_\omega$ for the time-fractional phase-field equations, which preserves the dissipation law under a restriction of the weight function.
Moreover, the time-fractional derivative of classical energy is proved to be nonpositive, which has been observed in previous numerical simulations \cite{du2019time}.

We remark that Theorem \ref{thm:kernel} on judging a positive definite kernel is innovative, which is based on the special Cholesky decomposition.
This result is the key ingredient in this article that allows us to analyze the dissipation property of weighted energy and the time-fractional derivative of classical energy. 


\bibliography{bibfile}
\bibliographystyle{unsrt}

\end{document}